\documentclass[12pt,reqno]{amsart}

\usepackage{amsmath, amssymb}
\usepackage[frame,cmtip,arrow,matrix,line,graph,curve]{xy}
\usepackage{graphpap,color,paralist,pstricks}
\usepackage[mathscr]{eucal}
\usepackage[pdftex]{graphicx}
\usepackage[pdftex,colorlinks,backref=page,citecolor=blue]{hyperref}
\usepackage{tikz}
\usepackage{kotex}
\usepackage{array}
\usepackage{pgfplots}
\usepackage{mathrsfs}
\usetikzlibrary{arrows} 
\usepackage{tikz-cd}
\usepackage{mathtools}

\newcolumntype{P}[1]{>{\centering\arraybackslash}p{#1}}
\newcolumntype{M}[1]{>{\centering\arraybackslash}m{#1}}

\newtheorem{theorem}{Theorem}[section]
\newtheorem{proposition}[theorem]{Proposition}

\newtheorem{lemma}[theorem]{Lemma}

\theoremstyle{definition}

\newtheorem{example}[theorem]{Example}

\newtheorem{remark}[theorem]{Remark}

\newtheorem{notation}[theorem]{Notation}

\newcommand{\PP}{\mathbb{P}}

\newcommand{\CC}{\mathbb{C}}
\newcommand{\FF}{\mathbb{F}}

\newcommand{\ZZ}{\mathbb{Z}}

\newcommand{\cO}{\mathcal{O} }

\newcommand{\cC}{\mathcal{C} }
\newcommand{\cE}{\mathcal{E} }
\newcommand{\cF}{\mathcal{F} }

\newcommand{\cH}{\mathcal{H} }
\newcommand{\cI}{\mathcal{I} }

\newcommand{\cK}{\mathcal{K} }
\newcommand{\cM}{\mathcal{M} }

\newcommand{\cR}{\mathcal{R} }
\newcommand{\cS}{\mathcal{S} }

\newcommand{\cU}{\mathcal{U} }

\newcommand{\enm}[1]{\ensuremath{#1}}

\newcommand{\cal}[1]{\mathcal{#1}}

\renewcommand{\bar}[1]{\overline{#1}}

\newcommand{\Ee}{\enm{\cal{E}}}

\newcommand{\Oo}{\enm{\cal{O}}}

\newcommand{\Rr}{\enm{\cal{R}}}

\newcommand{\rH}{\mathrm{H} }

\newcommand\bH{\mathbf{H}}
\newcommand\bK{\mathbf{K}}
\newcommand\bM{\mathbf{M}}

\newcommand\dual{^{\vee}}

\newcommand{\fsl}{\mathfrak{sl}}

\def\Sym{\mathrm{Sym} }
\def\Hom{\mathrm{Hom} }
\def\Ext{\mathrm{Ext} }

\def\Gr{\mathrm{Gr} }

\def\SL{\mathrm{SL}}

\def\lr{\rightarrow}

\newcommand{\rank}{\mathrm{rank}\, }

\newcommand{\ses}[3]{0\lr{#1}\lr{#2}\lr{#3}\lr 0}

\hypersetup{%
pdftitle={Rational curves in the Mukai-Umemura variety},%
pdfauthor={Kiryong Chung, Jaehyun Kim, and Jeong-Seop Kim},%
pdfkeywords={Mukai-Umemura variety, Rational curve, Torus fixed curve, Poincar\'e polynomial},%
citecolor=blue,%
linkcolor=blue,%
}

\begin{document}

\title{Rational quartic curves in the Mukai-Umemura variety}
\date{\today}

\author{Kiryong Chung}
\address{Department of Mathematics Education, Kyungpook National University, 80 Daehakro, Bukgu, Daegu 41566, Korea}
\email{krchung@knu.ac.kr}

\author{Jaehyun Kim}
\address{Department of Mathematics, Ewha Womans University, 52 Ewhayeodae-gil, Seodaemun-gu, Seoul, Korea}
\email{kjh6691@ewha.ac.kr}

\author{Jeong-Seop Kim}
\address{School of Mathematics, Korea Institute for Advanced Study, 85 Hoegiro Dongdaemun-gu, Seoul 02455, Republic of Korea}
\email{jeongseop@kias.re.kr}


\begin{abstract}
Let $X$ be the Fano threefold of index one, degree $22$, and $\mathrm{Pic}(X)\cong\mathbb{Z}$. Such a threefold $X$ can be realized by a regular zero section $\mathbf{s}$ of $(\bigwedge^2\mathcal{F}^{*})^{\oplus 3}$ over Grassmannian variety $\mathrm{Gr}(3,V)$, $\dim V=7$ with the universal subbundle $\mathcal{F}$. When the section $\mathbf{s}$ is given by the net of the $\mathrm{SL}_2$-invariant skew forms, we call it by the Mukai-Umemura (MU) variety. In this paper, we prove that the Hilbert scheme of rational quartic curves in the MU-variety is smooth and compute its Poincar\'e polynomial by applying the Bia{\l}ynicki-Birula's theorem.
\end{abstract}

\maketitle

\section{Motivation and results}
We work on the field $\CC$ of complex numbers.

\subsection{Motivation}
Let $X$ be the Fano threefold of index one, degree $22$, and $\text{Pic}(X)\cong\mathbb{Z}$. It is one of the minimal compactifications of the space $\CC^3$ such that the second Betti number is one (Hirzebruch’s question). While the moduli spaces of others (the proejctive space $\PP^3$, quadric threefold $Q_3$, and quintic del-Pezzo threefold $V_5$) are trivial, the variety $X$ has a six-dimensional moduli space. Among the automorphism groups of the varieties $X$, the largest dimensional one is $\mathrm{Aut}(X)\cong \mathrm{PGL}_2(\CC)$.
In this case, the variety $X$ has an $\SL_2$-representation theoretic description. In fact, through the result of Mukai and Umemura (\cite{MU83}), the variety $X$ fills up the missing one of Iskovskikh's classification (\cite{IP99}). On the other hand, it is very well-known that the threefold $X$ has been studied in various aspects: Derived category of algebraic varieties, Birational geometry, and K-stability of manifolds (\cite{Kuz97, KPS18, Don08}).

In the previous work of the first named author, the moduli space of rational curves was investigated in the spaces: $\PP^3$, $Q_3$, and $V_5$ (\cite{CK11, CCM16, Chu22, CHY21}).
He studied the birational relation among various compactifications of rational curves of low degrees in those spaces. To do such comparisons, it is a starting point to understand the geometry of the compactifications. So, in this paper, we will explicitly compute the weights of deformation space of the rational curves in $X$ to apply Bia{\l}ynicki-Birula's theorem (Section \ref{bbthm}). For such a purpose, we will use the descriptions of $X$ in terms of sub-locus of twisted cubics in $\PP^3$ as an $\SL_2$-equivariant setting which has been well-studied in \cite{Kuz97} and \cite{Sch91}.

\subsection{Results}
Let us fix a projective embedding $X\subset \PP^r$. Let $\bH_d(X)$ be the Hilbert scheme parameterizing curves $C$ in $X$ with Hilbert polynomial $\chi(\cO_C(m))=dm+1$. When $X$ is the MU-variety, the Hilbert scheme $\bH_1(X)$ is isomorphic to a non-reduced plane curve (Proposition \ref{linesinx}). One interesting point is that any two different lines in $X$ are disjoint. Hence the Hilbert scheme $\bH_2(X)$ of conics only parameterizes smooth conics or double lines (Proposition \ref{conicsinx}). For the detailed description of lines and conics for general $X$, see Section \ref{lowrat}. In this paper, we confirm the Hilbert scheme $\bH_3(X)$ of twisted cubics is smooth (Proposition \ref{h3sm}) and find the generators of the deformation space of a twisted cubics in $X$ (Section \ref{31}). It will help to study the moduli points of the space $\bH_4(X)$ (Section \ref{trofixpoin}). On the other hand, by the result of Kuznetsov (\cite{Kuz97}), there exists a flop diagram among two $\PP^1$-bundle over $\PP^3$ and $X$ respectively (Proposition \ref{flopd}). Using this diagram, we find the torus fixed \emph{irreducible} rational quartic curves in $X$. Since the reducible curve parameterized by $\bH_4(X)$ is a union of lower degree curves, it is not hard to find the torus fixed curve. Lastly, for the non-reduced curve, we will use the general result about a multiple structure of a Cohen-Macaulay curve in a smooth projective threefold. See the paper \cite{BF87} and \cite{NS03} for the general result and the computation in the case of the projective space $\PP^3$. Unlike the case $\PP^3$, it is known that the MU-variety does not contain any elliptic quartic curves (\cite[Corollary 3.9]{CFK23}). Hence we can conclude that all curves of degree $4$ in $X$ are Cohen-Macaulay (Proposition \ref{cucm}). After all, we have the following result.
\begin{theorem}[\protect{Proposition \ref{mainpr} }]\label{mainthm}
Let $X$ be the Mukai-Umemura threefold and $\bH_4(X)$ be the Hilbert scheme of degree $4$ and genus $0$ curves in $X$. Then 
\begin{enumerate}
\item $\bH_4(X)$ is smooth, $4$-dimensional variety and
\item the Poincar\'e polynomial of $\bH_4(X)$ is given by $P(\bH_4(X))=1+p+2p^2+p^3+p^4$.
\end{enumerate}
\end{theorem}

\subsection{Comments on Theorem \ref{mainthm}}
As mentioned in \cite[Remark 10]{Kuz97}, the space $\bH_4(X)$ is birational to the Grassmannian variety $\Gr(2,4)$. Our result may provide a clue to concretely describe the explicit global description of the Hilbert scheme $\bH_4(X)$. Also, in \cite[Proposition 4.12]{CS23}, the authors found $\CC^*\rtimes \ZZ_2$-fixed irreducible rational curves of degree $d\leq 12$ in $X$. On the other hand, from \cite[Corollary 3.9]{CFK23}, one can expect that BPS-invariant of genus $1$ over $X$ is $n_{1,0}=0$. So it may be a direct computation of the Donaldson-Thomas invariants on this local Calabi-Yau $4$-fold as did in \cite{CLW21}. To establish this one, we need to find a global description of $\bH_4(X)$ and its birational relation with the variety $\Gr(2,4)$. One of the natural approaches is to use the \emph{Sarkisov link} among $X$ and quintic del-Pezzo threefold $V_5$ (\cite{KPS18}).

\subsection{Streams of paper}
In Section 2, we review well-known results of the construction of $X$ or rational curves of degree $d\leq 3$ in $X$. In Section 3, we find a local chart of $X$ as a subvariety of Grassmannian variety and compute the $\CC^*$-fixed irreducible curve of degree $d=4$. In Section 4, with the help of Macaulay2 (\cite{M2}), we find the deformation space of quartic curves in the MU-variety and compute the Poincar\'e polynomial of $\bH_4(X)$. 

\subsection*{Notations and conventions}
\begin{itemize}
\item Let us denote by $\Gr(k,n)$ the Grassmannian variety parameterizing $k$-dimensional subspaces in a fixed vector space $V$ with $\dim V=n$.
\item When no confusion can arise, we sometimes do not distinguish the moduli point $[x]\in \cM$ from the object $x$ parameterized by $[x]$.
\item The set of fixed points of $X$ is denoted by $X^{\CC^*}$ under the $\CC^*$-action.
\end{itemize}

\subsection*{Acknowledgements}
The authors gratefully acknowledge the many helpful suggestions and comments of Jinwon Choi, Atanas Iliev, Shigeru Mukai, and Kyeong-Dong Park during the preparation of the paper.

\section{Preliminary}

\subsection{Rational curves of degree $\leq 2$ in $X$}\label{lowrat}
In this subsection, we recall the known result about the rational curves in the Fano threefold $X$ of index one, degree $22$, and $\text{Pic}(X)\cong\mathbb{Z}$.

\begin{proposition}[\protect{\cite[Corollary 5.2.14]{IP99}}]\label{linesinx}
The normal bundle of a line $L$ in $X$ is isomorphic to
\[
 N_{L/X}\cong \cO_L\oplus \cO_L(-1) \;\text{or}\;\cO_L(1)\oplus \cO_L(-2).
\]
Furthermore, 
\begin{enumerate}
\item for general $X$, the normal bundle of each line in $X$ is of the first type and $\bH_1(X)$ is isomorphic to a smooth plane quartic curve.
\item When $X$ is the MU-variety, the normal bundle of each line in $X$ is of the second type and $\bH_1(X)$ is isomorphic to a planar double conic (i.e., the square of a smooth conic).
\end{enumerate}
\end{proposition}
Let us define the \emph{planar double line} $L^2$ as the non-split extension sheaf $F$ ($\cong\cO_{L^2}$)
\[
\ses{\cO_L(-1)}{F}{\cO_L},
\]
where $L$ is a line. A double line $L^2$ in $X$ supported on $L$ is classified by $\Ext^1(\cO_L, \cO_L(-1))\cong \rH^0(N_{L/Y}(-1))$.
\begin{proposition}[\protect{\cite[Remark 5.2.15]{IP99}}]\label{conicsinx}
The Hilbert scheme $\bH_2(X)$ of conics is isomorphic to $\bH_2(X)\cong \PP(W_2)$, $\dim W_2=3$. For the MU-variety $X$, the vector space $W_2$ is isomorphic to the standard $\SL_2$-representation $W_2\cong \Sym^2\CC^2$ . Also, let $Q\subset \bH_2(X)$ be the $\CC^*$-invariant smooth conic. Then each point $[C]$ in $\bH_2(X)\setminus Q$ parameterizes a smooth conic $C$ with normal bundle $N_{C|X}=\cO_C\oplus \cO_C$. Each point $[C]\in Q$ parameterizes a planar double line $C=L^2$ in $X$.
\end{proposition}
For general $X$ (for example, (1) of Proposition \ref{linesinx}), since $N_{L/X}=\cO_L\oplus \cO_L(-1)$ there does not exist any planar double line in $X$. Hence the Hilbert scheme $\bH_2(X)$ for general $X$ parameterizes a smooth conic or a pair of lines (\cite[Corollary 4.3]{BF11}). Finally, let us mention a property of a rational Cohen-Macaulay (shortly, CM) curve which we will use later.
\begin{proposition}[\protect{\cite[Lemma 3.17]{CC11}}]\label{stabr}
Let $C$ be a CM-curve of genus $0$ in $\PP^r$. Then the structure sheaf $\cO_C$ is stable.
\end{proposition}
\subsection{A Bia{\l}ynicki-Birula's theorem}\label{bbthm}
Let $X$ be a smooth projective variety with $\CC^*$-action. Then the $\CC^*$-invariant loci of $X$ decompose into $X^{\CC^*}=\bigsqcup\limits_{i} Y_i$ such that each component $Y_i$ is connected. For each $Y_i$ the $\CC^*$-action on the tangent bundle $TX|_{Y_i}$ provides a decomposition as $TX|_{Y_i}=T^+\oplus T^{0}\oplus T^{-}$ where $T^+$, $T^{0}$ and $T^{-}$ are the subbundles of $TX|_{Y_i}$ such that the group $\CC^*$ acts with positive, zero and negative respectively. Under the local linearization, $T^0\cong TY_i$ and
\[
T^+\oplus T^{-}=N_{Y_i/X}=N^{+}(Y_i)\oplus N^{-}(Y_i)
\]
is the decomposition of the normal bundle of $Y_i$ in $X$. Let  $\mu^{\pm}(Y_i):=\rank N^{\pm}(Y_i)$. A fundamental result in the theory of $\CC^*$-action on $X$ has been provided by A. Bia{\l}ynicki-Birula. Let
\[
X^{+}(Y_i)=\{x\in X \mid \lim_{t\to 0} t\cdot x \in Y_i\} \; \text{and}\; X^{-}(Y_i)=\{x\in X \mid \lim_{t\to \infty} t\cdot x \in Y_i\}.
\]
\begin{theorem}[\protect{\cite{Bir73}}]\label{prop:onetoone}
Under the above assumptions and notations,
\begin{enumerate}
\item $X=\bigsqcup\limits_{i} X^{\pm}(Y_i)$ and
\item For each connected component $Y_i$, there are $\CC^*$-equivariant isomorphism $X^{\pm}(Y_i)\cong N^{\pm}(Y_i)$ over $Y_i$ where $N^{\pm}(Y_i)\lr Y_i$ is a Zariski locally trivial fiberation with the affine space $\CC^{\mu^{\pm}(Y)}$.
\end{enumerate}
\end{theorem}

\section{$\SL_2$-equivariant construction of the MU-variety}
From now on, unless otherwise stated, let $X$ be the Fano threefold of index one, degree $22$, and $\text{Pic}(X)\cong\mathbb{Z}$.
\subsection{Construction of the variety $X$ and a flat family of twisted cubics}\label{conxandcubic}
In this subsection, we recall constructions of $X$ as a sublocus of the Hilbert scheme of twisted cubics. Let $W$ be a vector space of $\dim W= 4$ and $W^*$ be the dual vector space of the space $W$. Let $\bH$ be the component of twisted cubic curves of the full Hilbert scheme $\bH_3(\PP^3)$. Let $\cC\subset \bH\times \PP^3$ be the universal subscheme parameterized by $\bH$. Lastly, let $p$ and $q$ be the projection map from $\bH\times \PP(W)$ to $\bH$ and $\PP(W)$ respectively. By twisting $\cO_{\PP^3}(2)$ of the universal structure sequence $\ses{I_{\cC}}{\cO_{\bH\times \PP^3}}{\cO_{\cC}}$ and push it down along the map $p:\bH\times\PP^3\lr\bH$, we get an exact sequence
\begin{equation}\label{univseq}
0\lr p_*I_{\cC}(2)\lr \Sym^2W^*\otimes \cO_{\bH}\lr p_*\cO_{\cC}(2)\lr 0.
\end{equation}
The direct image sheaf $\cF:=p_*I_{\cC}(2)$ is a vector bundle of rank $3$ on $\bH$ which induces a birational morphism
\[
\phi: \bH\lr \Gr(3, \Sym^2W^*).
\]
Let $\bK:=\phi(\bH)$ be the image of $\bH$ by the map $\phi$. Then $\bK$ is so called, the \emph{Kronecker's modules} space. Note that the map $\phi$ is an isomorphism over the space of CM-curves parameterized by $\bH_3(\PP^3)$. From the second map of the exact sequence of \eqref{univseq} (tensoring with $W^*$) and the multiplication map $\Sym^2W^*\otimes W^* \lr\Sym^3W^*$ over $\bK$, we have a sheaf homomorphism
\begin{equation}\label{kebu}
\Phi: \cF\otimes  W^*\lr \Sym^3W^*\otimes \cO_{\bK}.
\end{equation}
\begin{notation}[\protect{\cite[Section 3]{ES95}}]\label{cebu}
Let $\cE:=\ker(\Phi)$ be the kernel bundle of the map $\Phi$ in \eqref{kebu}.
\end{notation}
The fiber of $\cE$ at the closed point $\phi([C])\in \bK$ is generated by the syzygy relation among the quadric generators of the twisted ideal $I_{C}(2)$ of the curve $C\subset \PP^3$. Furthermore, the image of the map $\Phi$ is $\text{im}(\Phi)=p_*I_{\cC}(3)$ and it gives rise to a twisted universal family
\begin{equation}\label{univcubics}
0\lr \cE \lr \cF\boxtimes \cO_{\PP^3}(1)\lr \cO_{\bK}\boxtimes  \cO_{\PP^3}(3)\lr \cO_{\cC}(3)\lr0
\end{equation}
of the universal curve $\cC\subset\bK\times \PP^3$. On the other hand, for a net of quadric forms $q\subset \Sym^2 W$, let $q^{\bot}\subset   \Sym^2 W^*$ be the space annihilated by the net $q$ with respect to a polar pairing $\Sym^2W\otimes \Sym^2W^*\lr \CC$ (\cite[Section 2]{Sch91}). Then the ideal sheaf $\widetilde{I}_{q^{\bot}}$ generated by $q^{\bot}$ has a free resolution on $\PP^3=\PP(W)$ where the first truncated map is given by
\begin{equation}\label{resGor}
\cO_{\PP^3}^{\oplus 8}(-3)\oplus \cO_{\PP^3}^{\oplus 3}(-4)\stackrel{(\phi_1,\phi_2)}{\longrightarrow} \cO_{\PP^3}^{\oplus 7}(-2) \lr \widetilde{I}_{q^{\bot}}\lr 0.
\end{equation}
\begin{notation}[\protect{\cite[Proposition 4.2]{Sch91}}]\label{noinst}
The twisted kernel $\cR:=\text{ker}(\phi_1)\otimes \cO_{\PP^3}(5)$ of the map $\phi_1$ in \eqref{resGor} is called by the \emph{instanton bundle} on $\PP^3$.
\end{notation}
 Furthermore, tor-algebra structure of the resolution of $\widetilde{I}_{q^{\bot}}$ in \eqref{resGor} induces a net of alternative form on $q^{\bot}$ which we will denote by $\omega$ (For the detail, see \cite[Section 5]{Sch91}). We let $X$ be the \emph{tri-isotropic} Grassimannian variety $\Gr^{\omega}(3, q^{\bot})\subset \Gr(3, \Sym^2W^*)$, which is our prime Fano $3$-fold.
Note that the restriction of the universal family $\cC$ on $X$ provides a flat family of twisted cubics in $X$ parameterized by $\PP^3$ (\cite[Theorem 2.4]{KS04}). Furthermore,
\begin{proposition}[\protect{\cite{Kuz97}, \cite[Theorem 2.4]{KS04}}]\label{h3sm}
Let $\bH_3(X)$ be the Hilbert scheme of twisted cubics in $X$. Then $\bH_3(X)\cong \PP^3$.
\end{proposition}
\begin{proof}
From the restricted universal sequence \eqref{univcubics}, one can see that $\dim\Hom_{X}(I_C, \cO_C)=3$ from the fact that the stability of $\cE$, $\cF$ and $\Ext_{X}^i(\cF,\cE)=0$, $i\geq 0$. Note that $\cE$, $\cF$ are two components of exceptional collections of $X$.
\end{proof}
To find a $\CC^*$-fixed quartic curve in $X$, we need to know some explicit generators of the deformation space of $\bH_3(X)$. Hence we will compute the deformation space in various local charts in Section \ref{31}.
\subsection{Review on the paper \cite{Kuz97}}
In this subsection, we review the results of \cite{Kuz97}, which we will use later to find an irreducible $\CC^*$-fixed rational curves. In the paper \cite{Kuz97}, by restricting the flip diagram (obtained by the geometric invariant theory \cite{Tha96}) among the birational models of the universal cubic curves on $\PP^3$, the author obtained an explicit flop diagram among two $\PP^1$-bundles over $\PP^3$ and $X$ respectively. Roughly speaking, two bundles are families of the (rational) maps of degree three of both spaces (\cite[Section 4.4]{Kuz97} and \cite[Proposition 4.8]{CM17}).
\begin{proposition}[\protect{\cite[Theorem 1 and Section 4.1.2]{Kuz97}}]\label{flopd}
Let $\pi_{-}:\PP(\cR)\lr \PP^3$ and $\pi_{+}:\PP(\cE)\lr X$ be the canonical projection maps for the bundles $\cR$ and $\cE$ in Notation \ref{noinst}, \ref{cebu} respectively. Then there exists a standard flop diagram among two projective bundles $\PP(\cR)$ and $\PP(\cE)$:
\begin{center}
\begin{minipage}[c]{0.4\textwidth}
\[\xymatrix@C2pc@R2pc{
& \cM \ar[ld]_{\hat{\pi}_{-}} \ar[rd] ^{\hat{\pi}_{+}} & \\
\PP(\cR) \ar[dr] & & \PP(\cE) \ar[dl] \\
& Y &
}\]
\end{minipage}
\begin{minipage}[c]{0.1\textwidth}
\[\supset\qquad\vspace{0.7em}\]
\end{minipage}
\begin{minipage}[c]{0.4\textwidth}
\[\xymatrix@C0.5pc@R1pc{
& \PP(\cS^-)\times_{C}\PP(\cS^+) \ar[ld]\ar[rd] & \\
\PP(\cS^-) \ar[rd] & & \PP(\cS^+) \ar[ld] \\
& C &
}\vspace{0.7em}\]
\end{minipage}
\end{center}
where the flopping loci $\PP(\cS^{\pm})$ are projectivized $\PP^1$-bundle of some rank two bundles $\cS^{\pm}$ over a plane quartic curve $C$ (embedded in $\PP^3$ and $X$ respectively). Then
\begin{enumerate}
\item the master space $\cM$ is the resolution $\cM\lr \cC$ of the universal curve $\cC\subset \PP^3\times X$.
\item Let $\widetilde{\pi}_{\pm}=\hat{\pi}_{\pm}\circ \pi_{\pm}$. For the classes $h={\widetilde{\pi}_{-}}^*\cO_{\PP(W)}(1)$, $H={\widetilde{\pi}_{+}}^*\cO_X(1)$, $r={\hat{\pi}_0}^*\Oo_{\PP(\Rr)}(1)$, $e={\hat{\pi}_0}^*\Oo_{\PP(\Ee)}(1)$ and $d=[\cM]$, we have equalities:
\begin{equation}\label{nclass}
h=3e-H-d,\; H=3r+5h-d
\end{equation}
regraded as a divisor class in the master space $\cM$.
\item The flopping base space $Y\subset \PP^7$ is the image of the map given by the complete linear system $|r+2h|(\rH^0(\PP(\cR), \cO(r+2h))\cong \rH^0(\PP^3, \cR(2)))$ on $\PP(\cR)$.
\end{enumerate}
\end{proposition}
\begin{notation}
Let us denote by $S^{-}\subset \PP^3$ (resp. $S^{+}\subset X$) the image of the flopping locus $\PP(\cS^{-})$ (resp. $\PP(\cS^{+})$) along the canonical projection map $\pi_{-}$ (resp. $\pi_{+}$) in Proposition \ref{flopd}.
\end{notation}
When $X$ is the MU-variety, then the image variety $S^{\pm}$ (of the reduced scheme structure) is a tangentially developable surface in each space $\PP^3$ and $X$ of the degree $3$, $12$ rational normal curve $C_{\text{red}}$ (\cite{Kuz97, MU83}) respectively. On the other hand, for general $X$ (for example, $\bH_1(X)$ is a smooth variety), the type of the cubic curves in $X$ are classified by the strata: $\bH_3(X)\cong\PP^3=(\PP^3\setminus S^{-})\cup (S^{-}\setminus C)\cup C$ where the closed point of each stratum parameterizes a twisted cubic, smooth conic with an attached line, and pair of lines with a multiplicity two structure respectively. On the other hand, when $X$ is the MU-variety, the deepest stratum $C_{\text{red}}$ parameterizes triple lines lying on a quadric cone (\cite[page 41]{EH82}), and the others are the same one for general variety $X$. For the study of cubic curves in $\PP^3$ parameterized by $X$, see Section 3 in \cite{Mar04}.

\begin{lemma}\label{setmapprpe}
Let $X$ be the MU-variety. Let $q: \PP(\cR)\dashrightarrow \PP(\cE)$ be the flop map in Proposition \ref{flopd}. Let $l_p=\PP(\cR|_p)$ be the fiber over a point $p\in \PP^3$ and $\bar{l}_p=\pi_{+}(q(l_p))$. Then for each point $p\in \PP^3\setminus S^{-}$, $\bar{l}_p$ is a twisted cubic curve. For $p\in S^{-}\setminus C_{\text{red}}$, $\bar{l}_p$ is a smooth conic with a base point. Lastly, $\bar{l}_p$ is a point for $p\in C_{\text{red}}$.
\end{lemma}
\begin{proof}
All of the claims were studied in Section 4.4 of \cite{Kuz97} except the last case. If $p\in C_{\text{red}}$, then the possible image $\bar{l}_p$ is a line or point since the point $p\in C_{\text{red}}\subset \PP^3$ parameterizes a triple line in $X$. If it is a line, then it lies in $S^{+}$ and thus it comes from a fiber $\PP(\cS^{+})$ (\cite[Section 4.2]{Kuz97}). But the image of the line $l_p$ along the contraction $\PP(\cR)\lr Y$ is not a point in $Y$ by (3) of Proposition \ref{flopd}. This is a contradiction and thus we proved the claim.
\end{proof}

\subsection{Explicit presentation of $\SL_2$-invariant alternative forms}
From now on, we focus on the MU-variety $X$. We will explicitly construct a family of rational cubic curves in $X$ parameterized by $\PP^3$ to find irrational rational quartic curves in $X$ fixed by the maximal torus ($\cong \CC^*$) under the $\SL_2$-action. In the MU-variety case, by the work of S. Donaldson (\cite[Section 5]{Don08}), the choice of the net of quadric forms $q$ and hence the net of alternative forms $\omega$ are canonical ones, and all of them are $\SL_2$-invariant ones. To find the explicit forms, we need to use the $\SL_2$-representation theory (\cite{FH91} and \cite[Section 5.5]{CS15}). Let us fix $x$, $y$ be the basis of $\CC^2$. Let $W_{d}=\Sym^d\left(\CC^{2}\right)$ with the $\SL_{2}$-action induced from the left multiplication of $\SL_2$ on $\CC^{2}$.
Let 
\[
\left\{u_{d}=x^{d},\; u_{d-2}=x^{d-1}y,\; ... \;, u_{-1}=xy^{2-d},\; u_{-d}=y^{d}\right\}
\]
be the standard basis of $W_d$. The infinitesimal $\fsl_{2}$-action on $W_{d}$ is given by 
\[
e=x\partial_{y},\quad f=y\partial_{x},\quad h=[e,f]
\]
for the standard basis $\langle e,f,h\rangle $ for $\fsl_{2}$. For $0\leq i\leq d$,
\[
\left\{\hspace{.3cm}
\begin{split}
e{\cdot} u_{d-2i}&= iu_{d-2(i-1)},\\
f{\cdot} u_{d-2i}&=(d-i)u_{d-2(i+1)} ,\\
h{\cdot} u_{d-2i}&= (d-2i)u_{d-2i}.
\end{split}\right.
\]
Note that for $g \in \fsl_{2}$, $g\cdot(u_{i}\wedge u_j)=(g{\cdot} u_i)\wedge u_j +u_i\wedge (g{\cdot} u_j)$ for $i\neq j$. Let us fix the standard $\SL_2$-equivariant basis by
\[
u_3=x^3, u_1=x^2y, u_{-1}=xy^2, u_3=y^3
\]
of the dual space $W_3=W^*$ in subsection \ref{conxandcubic}. Let $u_3^*$, $u_1^*$, $u_{-1}^*$, $u_{-3}^*$ be the standard basis of the dual space $W_3^*=W$ such that $u_i^*(u_j)=\delta_{ij}$ for $i, j=\pm 1, \pm 3$. One can easily see that $e$ and $f$ act on the space $W_3^*$ by
\[\begin{split}
e{\cdot}u_{-3}^*=0&,\, e{\cdot}u_{-1}^*=3u_{-3}^*,\, e{\cdot}u_{1}^*=2u_{-1}^*,\, e{\cdot}u_{3}^*=u_1^*,\\
f{\cdot}u_{-3}^*=u_{-1}^*&,\,f{\cdot}u_{-1}^*=2u_{1}^*,\, f{\cdot}u_{1}^*=3u_{3}^*,\, f{\cdot}u_{3}^*=0.
\end{split}\]
On the other hands, let us identify $\CC^2\cong \CC^{2*}$ by $x\mapsto \partial_x$ and $y\mapsto \partial_y$. Let
\[
\partial_3:=\partial_y^3, \partial_1:=\partial_y^2\partial_x, \partial_{-1}:=\partial_y\partial_x^2, \partial_{-3}:=\partial_x^3
\]
be the basis of the symmetric space $\Sym^3\CC^{2*}$ with weights $3$, $1$, $-1$, $-3$. Then the $\SL_2$-equivariant identification $W_3^*\cong\Sym^3\CC^{2*}$ is provided by
$6u_{-3}^*=\partial_3, 2u_{-1}^*=\partial_1, 2u_{1}^*=\partial_{-1}, 6u_{3}^*=\partial_{-3}$.
\begin{lemma}\label{invar7dim}
Under the above assumptions and notations,
the unique $7$-dimensional $\fsl(2)$-invariant subspace in $\Sym^2W_3(\cong \Sym^6\CC^2\oplus \Sym^3\CC^2)$ is generated by
\begin{equation}\label{vectori6}
U_6:=\langle u_{3}^2, u_{3}u_1, 3u_1^2+2u_{3}u_{-1}, 9u_1u_{-1}+u_{3}u_{-3}, 3u_{-1}^2+2u_1u_{-3}, u_{-1}u_{-3}, u_{-3}^2\rangle.
\end{equation}
\end{lemma}
\begin{proof}
Clearly, $u_3^2$ are unique (up to the scalar multiplication) highest weight vector in $\Sym^2W_3$. The action of $f$ generates the list:
\[
\frac{f{\cdot} u_3^2 }{6}=u_3u_{1}, f{\cdot}u_3u_1=3u_1^2+2u_3u_{-1},  f{\cdot}(3u_1^2+2u_3u_{-1})=9u_1u_{-1}+u_{3}u_{-3},\cdots, f{\cdot}u_{-1}u_{-3}=u_{-3}^2.
\]
\end{proof}
\begin{remark}
\begin{enumerate}
\item The $3$-dimensional invariant component $U_3$ of $\Sym^2 W_3$ is the space of quadrics containing the $\SL_2$-invariant twisted cubic $C_{\text{red}}=C_3\subset \PP(W_3^*)$ (\cite[Exercise 11.20, 11.32]{FH91}). In fact, $I_{C_3}=U_3=\langle u_1^2-u_3u_{-1}, u_1u_{-1}-u_3u_{-3},u_{-1}^2-u_{-3}u_{1}\rangle$.
\item The defining equation of the tangentially developable surface $S^{-}$ of the twisted cubic $C_3$ is (\cite[Section 2]{EL19})
\[
4u_{-3}u_1^3-6u_3u_1u_{-1}u_{-3}+4u_3u_{-1}^3-3u_1^2u_{-1}^2+u_3^2u_{-3}^2=0.
\]
\end{enumerate}
\end{remark}
For later use, we address a property of the restriction of the instanton bundle $\cR$ (Notation \ref{noinst}) along a line $L$ in $\PP^3$. It is well-known that the restricted bundle of $\cR$ is of the form $\cR|_L\cong \cO_L(-j)\oplus \cO_L(j)$, which we call $L$ by a \emph{$j$-jumping} line.
\begin{proposition}
A line $L\subset \PP^3$ is a $2$-jumping line if and only if $L\subset S^{-}$ (i.e., $L$ is a tangent line of the twisted cubic $C_3$). If the multiplicity type of the intersection $L\cap S$ of $L$ and $S$ is $(3,1)$, then $L$ is an $1$-jumping line. Otherwise, all of the lines are $0$-jumping lines.
\end{proposition}
\begin{proof}
It is straightforward to check the properties of the jumping lines by the defining resolution of $\cR$ in \eqref{resGor}.
\end{proof}
\begin{example}\label{kjump}
Let us denote by $l_{i, j}$ the line supported on the $i$ and $j$-the coordinates of $\PP^3$. The line $l_{3,1}$ and $l_{-1,-3}$ are a $2$-jumping line. $l_{3,-1}$ and $l_{1,-3}$ are an $1$-jumping line. But the lines $l_{3,-3}$ and $l_{1,-1}$ are a $0$-jumping line. 
\end{example}

On the other hand, to find the net $q$ of quadrics in $\Sym^2W_3^*$, we need to be careful about choosing the basis. To avoid the confusion, we pick up the basis of $W_3^*$ as $u_{i}^*$. Then the invariant $3$-dimensional net in $\Sym^2W_3^*$ is given by
\[q=\langle 3u_{-3}^*u_{1}^*-u_{-1}^{*2},\;9u_{3}^*u_{-3}^*-u_{1}^*u_{-1}^*,\;3u_{-1}^*u_{3}^*-u_{1}^{2*}\rangle\]
since $e{\cdot}(3u_{-3}^*u_{1}^*-u_{-1}^{*2})=0$ and the action of $f$ on the highest weight vector generates the other ones. Now the polar pairing $\Sym^2W_3^*\otimes \Sym^2W_3\lr \CC$
is given by \[u_{\bold{k}}^{*}\otimes u_{\bold{m}}\mapsto\delta_{\bold{km}} {2 \choose \bold{k}}^{-1}=\frac{1}{2}\delta_{\bold{km}}\partial_{u_{\bold{m}}}(u_{\bold{m}})\]where $\bold{k}$ and $\bold{m}$ are multi-indexed. 
The polar ideal $q^{\bot}$ of the net $q$ of quadrics under this pairing is exactly the space $q^{\bot}=U_6$ in \eqref{vectori6} of Lemma \ref{invar7dim}, which matches with the construction in Section 4 of the paper \cite{Sch91}.

Returning to the problem of the construction of $X$ as the sublocus of the Hilbert scheme of twisted cubics in $\PP(W_3^*)$, let us fix the ordered basis of $U_6$ in \eqref{vectori6}:
\begin{equation}\label{ordw}
\begin{split}
w_6:&=u_3^2, w_4:=u_{3}u_1, w_2:=3u_1^2+2u_{3}u_{-1}, w_0:=9u_1u_{-1}+u_{3}u_{-3},\\ 
w_{-2}:&=3u_{-1}^2+2u_1u_{-3}, w_{-4}:=u_{-1}u_{-3}, w_{-6}:=u_{-3}^2.
\end{split}
\end{equation}
There exists an induced $\fsl_{2}$-action on the basis $w_i$ as follows.
\begin{gather*}
f{\cdot}w_6=6w_4, f{\cdot}w_4=w_2, f{\cdot}w_2{\cdot}=2w_0, f(w_0)=6w_{-2}, f{\cdot}w_{-2}=10w_{-4}, f{\cdot}w_{-4}=w_{-6}, f{\cdot}w_{-6}=0,\\
e{\cdot}w_6=0, e{\cdot}w_4=w_6, e{\cdot}w_2=10w_4, e{\cdot}w_0=6w_{2}, e{\cdot}w_{-2}=2w_{0}, e{\cdot}w_{-4}=w_{-2}, e{\cdot}w_{-6}=6w_{-4}.
\end{gather*}
Let $w_i^*$ be the standard dual basis of $U_6$. Note the weight is reversed. Then
\begin{proposition}\label{altsl2}
The unique $3$-dimensional $\fsl_{2}$-subrepresentation $N$ of $\bigwedge^2 U_6^*$ are generated by
\begin{equation}\label{altform}
\begin{array}{cc}
\frac{3}{5}w_{-6}^*\wedge w_{4}^*-w_{-4}^*\wedge w_{2}^*+6w_{-2}^*\wedge w_{0}^*,&\; -\frac{1}{5}w_{-4}^*\wedge w_{4}^*+\frac{9}{5}w_{-6}^*\wedge w_{6}^*+w_{-2}^*\wedge w_{2}^*,\\
 -w_{-2}^*\wedge w_{4}^*+\frac{3}{5}w_{-4}^*\wedge w_{6}^*+6w_{0}^*\wedge w_{2}^*.&
\end{array}
\end{equation}
\end{proposition}
\begin{proof}
The highest weight vector $w^*$ in $\eta$ is of the form $w^*=a w_{-6}^*\wedge w_{4}^*+b w_{-4}^*\wedge w_{2}^*+c w_{-2}^*\wedge w_{0}^*$. Since $e{\cdot}w^*=0$, we get $10a+6b=6b+c=0$. Therefore one can choose $a=\frac{3}{5}$, $b=-1$, $c=6$. By the action of the operator $f$, we obtain
\[
\frac{1}{2}f{\cdot}w^*=-\frac{1}{5}w_{-4}^*\wedge w_{4}^*+\frac{9}{5}w_{-6}^*\wedge w_{6}^*+w_{-2}^*\wedge w_{2}^*,\; \frac{1}{2}f^2{\cdot}w^*=-w_{-2}^*\wedge w_{4}^*+\frac{3}{5}w_{-4}^*\wedge w_{6}^*+6w_{0}^*\wedge w_{2}^*.
\]
\end{proof}

\begin{remark}
Matrix form of the net $\eta$ of alternative forms in Proposition \ref{altsl2} is
\[\eta=
\begin{pmatrix}0& 0& 0& 0& 0& -\frac{3}{5}y_2& -\frac{9}{5}y_0\\
0& 0& 0& 0& y_2& \frac{1}{5}y_0& -\frac{3}{5}y_{-2}\\
0& 0& 0& -6y_2& -y_0& y_{-2}& 0\\
0& 0& 6y_2& 0& -6y_{-2}& 0& 0\\
0& -y_2& y_0& 6y_{-2}& 0& 0& 0\\
\frac{3}{5}y_2& -\frac{1}{5}y_0& -y_{-2}& 0& 0& 0& 0\\
\frac{9}{5}y_0& \frac{3}{5}y_{-2}& 0& 0& 0& 0& 0
\end{pmatrix}
\]
for the dual basis $y_2$, $y_{0}$, $y_{-2}$ of the ordered generators \eqref{altform} of the net $N\subset \bigwedge^2 U_6^*$.
\begin{enumerate}
\item To find the net of alternative forms on $U_6$, we should study the tor-algebra structure of the Artinian-Gorenstein ring
\[\CC[u_{3},u_{1},u_{-1}, u_{-3}]/q^{\bot}\]
where the ideal $q^{\bot}$ is generated by $U_6$ (\cite[Section 5]{Sch01}). One can show that the net matches with $\eta$ by using Macaulay2 (\cite{M2}).
\item The Pfaffians of the principle $6\times 6$ minor matrix of the net $\eta$ is generated by
\[
\langle 
y_2^3, y_2^2y_0, y_2y_0^2+y_2^2y_{-2}, y_0^3+6y_2y_0y_{-2}, y_{0}^2y_{-2}+y_{2}y_{-2}^2, y_0y_{-2}^2, y_{-2}^3
\;\rangle.
\]
Hence the polar form of the pairing: $\Sym^3 N\otimes \Sym^3 N^*\lr \CC$ is the double conic:
\begin{equation}\label{doublecon}
F:=(y_2y_{-2}+\frac{1}{4}y_{0}^2)^2=0.
\end{equation}
\end{enumerate}
\end{remark}
\begin{proposition}[\protect{\cite[Section 5]{Sch01}, \cite[Lemma 4.1]{Fae14}}]\label{lacx}
The Hilbert scheme of lines (resp. conics) in the MU-variety $X$ is isomorphic to $\bH_1(X)\cong \{F=0\}$ in \eqref{doublecon} (resp. $\bH_2(X)\cong \PP(N)$). Hence the $\CC^*$-fixed lines and conics are isolated ones.
\end{proposition}

\section{Hilbert scheme of rational curves of degree $4$ in $X$}
In section \ref{31}, we present the deformation space of twisted cubics in $X$ for the study of $\bH_4(X)$. Since the maximal torus action on $\bH_3(X)\cong \PP^3$ has four fixed points, it is enough to consider these cases only. In the subsection \ref{42}, we prove that $\bH_4(X)$ is smooth and obtain the Poincar\'e polynomial by computing the weights of tangent space.

Since the maximal torus action on Grassimannian $\Gr(3,U_6)$ is inherited from the $\SL_2$-action on $U_6$, it is well-defined on the MU-variety $X$. For the standard Schubert cell $V_{k}^{\circ}$ in Grassmannian variety $\Gr(3,U_6)$, let $V_{k}=X\cap V_{k}^{\circ}$ be the open chart in $X$. One can easily compute the weight and thus the variety $X$ has four affine charts around the fixed point in $\PP(\wedge^3U_6)$ as follows.
\begin{align*}
p_{12}=&\PP(w_6\wedge w_4\wedge w_2),\;p_{10}=\PP(w_6\wedge w_4\wedge w_2),\\
\;p_{-10}=&\PP(w_0\wedge w_{-4}\wedge w_{-6}), p_{-12}=\PP(w_{-2}\wedge w_{-4}\wedge w_{-6}).
\end{align*}
Because of the symmetry, it is enough to consider the first two affine charts around $p_{12}$ and $p_{10}$. Let
\[
P_3=
\begin{pmatrix}1&0&0&a_1&a_2&a_3&a_4\\0&1&0&a_5&a_6&a_7&a_8\\0&0&1&a_9&a_{10}&a_{11}&a_{12}\end{pmatrix}
\]
be the affine chart of $\Gr(3, U_6)$ around the point $p_{12}$ (similarily $p_{10}$). Then the affine chart of $X$ is given by the condition $P_3\eta P_3^T=O$ and thus $V_{12}\subset X$ is defined as the nine equations:
\begin{equation}\label{eqV12}
\begin{split}
\langle 5a_2+3a_7, 
a_3+9a_8,&
2a_2a_5-2a_1a_6+1/5a_4,
10a_1-a_{11}, 5a_2-9a_{12},\\
6a_2a_9-6a_1a_{10}-a_3,& 6a_5+a_{10}, 5a_6+a_{11}, 2a_6a_9-2a_5a_{10}-1/3a_7-1/5a_{12}\rangle.
\end{split}
\end{equation}
Also the ordered basis of $U_6$ in \eqref{ordw} provides the defining equation of the universal curve $\cC\subset  \PP^3\times V_{12}$. That is, the universal curve $\cC$ is defined by the relations \[P_3\begin{pmatrix}w_6&w_4&\cdots&w_{-4}&w_{-6}\end{pmatrix}^T=O\;\text{and}\;P_3\eta P_3^T=O.\] Note that the universal curve $\cC$ can be written as the maximal minors of the $3\times 2$ matrix:
\begin{equation}\label{equnivV12}
\begin{pmatrix}6u_1-a_{10}u_{-3}&12u_{-1}+6a_9u_{-3}\\ -2u_3-3a_{10}u_{-1}-a_{11}u_{-3}&
6u_1+18a_9u_{-1}+5a_{10}u_{-3}\\
-3a_{11}u_{-1}-12a_{12}u_{-3}&-10u_3+15a_{10}u_{-1}+4a_{11}u_{-3}
\end{pmatrix}
\end{equation}
which is similar to the presentation of twisted cubics in the case of the projective space (cf. \eqref{univcubics}).
\subsection{Generators of the deformation space of the Hilbert scheme $\bH_3(X)$}\label{31}
It is very well-known that the deformation space of the Hilbert scheme is
\[
\rH^0(X, \cH om (\cI_C, \cO_C))=\Hom_{X} (\cI_C, \cO_C)
\]
for a curve $C\subset X$.\\
Hence we will compute the deformation space $\Hom (\cI_C, \cO_C)$ and its weights by $\CC^*$-action for each affine chart of $X$. Since the $\CC^*$-action on $\bH_3(X)=\PP(W_3^*)$ is a standard one, the fixed points are coordinates points of $\PP(W_3^*)$. Geometrically, it parameterizes two types of curves: Triple line and smooth conic with a line. Clearly, the relevant line and conic are also $\CC^*$-fixed ones.

Let us denote by the coordinate points of $\PP^3$ as $p_{3}=[1:0:0:0]$, $p_{1}=[0:1:0:0]$, $p_{-1}=[0:0:1:0]$, and $p_{-3}=[0:0:0:1]$.
\begin{proposition}\label{deftrip}
Let $C$ be the triple line parameterized by the fixed point $p_{-3}\in \bH_3(X)=\PP(W_3^*)$. Then the generators of the deformation of $\bH_3(X)$ at $[C]$ are given by the weights: $6$, $4$ and $2$.
\end{proposition}
\begin{proof}
Since we know that the Hilbert scheme $\bH_3(X)$ is smooth, it is enough to collect the generators of given weights.
From the equations \eqref{eqV12} and \eqref{equnivV12}, one can read that the defining ideal of $C$ is given by $\langle a_4, a_8, a_{12}\rangle$ over the open subset $V_{12}$ (resp. the ideal $\langle a_{12}, a_8, a_7, 5a_6+a_{11}, 6a_5+a_{10}, a_4, a_3, a_2, 10a_1-a_{11}, a_{11}^2, a_{10}a_{11}, 5a_{10}^2-6a_9a_{11}
\rangle$ over $V_{12}^{\circ}$). By using Macaulay2 (\cite{M2}), one can show that the sheaf $\cH om (\cI_C, \cO_{C})$ over $V_{12}$ associates to
\[
\cH om (\cI_C, \cO_{C})(V_{12})=(\bigoplus_{i\neq 3, 6} \psi_i) \oplus (\CC[a_9,a_{10}, a_{11}]/\langle  a_{10}^2,a_{10}a_{11}, a_{11}^2\rangle)(\psi_3\oplus\psi_6)
\]
such that the generators $\psi_i$, $1\leq i \leq 6$ are given the column vectors
\begin{equation}\label{eq:gensall1}
\begin{pmatrix}
0&-3a_{11}&-5a_{10}&0&0&-2a_{11}\\
a_{11}&a_{10}&2a_9&0&0&0\\
0&0&0&a_{10}&a_{11}&5a_{9}
\end{pmatrix}
\end{equation}
with respect to the ordered generators $\{a_4, a_8, a_{12}\}$ of the ideal $\cI_C|_{V_{12}}$. On the other hand, around the fixed point $p_{10}$, let us denote $V_{10}^{\circ}$ by the $3\times 7$ matrix $Q_3$ of $b_i$, ($1\leq i\leq 12$) entries with proper positions. Then the ideal of $C$ is given as $\langle b_4, b_8, b_{12}\rangle$ over the open subset $V_{10}$ and thus the sheaf $\cH om (\cI_C, \cO_{C})$ over $V_{10}$ is
\[
\cH om (I_C, \cO_{C})(V_{10})=(\CC[b_9,b_{10}, b_{11}]/\langle b_{10}^2, b_{10}b_{11}, b_{11}^2\rangle) \phi_1 \bigoplus_{i=2}^{4} (\CC[b_9])\phi_i
\]
such that the generators $\phi_i$, $1\leq i \leq 4$ are given the column vectors
\begin{equation}\label{eq:genshom31}
\begin{pmatrix}\phi_1 &\phi_2 &\phi_3 &\phi_4 \end{pmatrix}=\begin{pmatrix}
0&5b_{10}&0&-3b_{11}\\
0&-2&b_{11}&b_{10}\\
1&0&0&0
\end{pmatrix}
\end{equation}
with respect to the ordered generators $\{b_4, b_8, b_{12}\}$. Surely, the coordinate change among the matrices $P_3$ and $Q_3$ is nothing but the elementary operation of matrices. By considering the $\CC^*$-weights of the deformation space in each chart, one can see that the global Hom-space $\Hom(\cI_C, \cO_C)$ is generated by extending the following column vectors in $V_{12}$:
\begin{equation}\label{genssec3}
\begin{pmatrix} a_1&0&0\\
a_5 &-\frac{1}{5}a_{11}&0\\
a_{9}&a_{10}&a_{11}
\end{pmatrix}=\begin{pmatrix} \frac{a_{11}}{10}&0&0\\
-\frac{a_{10}}{6} &-\frac{1}{5}a_{11}&0\\
a_{9}&a_{10}&a_{11}
\end{pmatrix}.
\end{equation}
For example, the first column $\phi_1$ over the open chart $V_{10}$ (of degree $1$ part) in \eqref{eq:genshom31} becomes the vector $-\frac{1}{6}\psi_2+\frac{1}{5}\psi_6$ in \eqref{eq:gensall1}. But the other generators $\phi_i$, $2\leq i\leq 4$ of any degrees can not be extended into the chart $U_{12}$. Since the weight of $\CC^*$-action are compatible with Pl\"{u}cker embedding of $X\subset\PP(\wedge^3U_6)$, the weights of the columns in \eqref{genssec3} are $4$, $8$, and $6$.
\end{proof}
\begin{proposition}[\protect{\cite[cf. Proposition 3.5]{CL17}}]\label{smlq}
Let $C$ be the curve parameterized by the fixed point $p_{-1}\in \bH_3(X)=\PP(W_3^*)$. Then,
\begin{enumerate}
\item the curve $C=L\cup Q$ consists of the union of the line $L$ and the conic $Q$.
\item \begin{equation}\label{lqn}N_{C/X}|_L\cong \cO_L(1)\oplus \cO_L(-1).\end{equation}
\item The weight of the tangent space of Hilbert scheme $\bH_3(X)$ at $[C]$ are $4$, $2$ and $-2$.
\end{enumerate}
\end{proposition}
\begin{proof}
(1) The result comes from the equation of the restriction of the universal curve $\cC\subset V_{10}\times \PP^3$.
The ideal of the curve $C$ is generated by $\langle b_1,b_2,b_3+9b_8,b_4,b_5,b_6,b_7,b_8b_9,b_{10},b_{11},b_{12}\rangle$ over $V_{10}^\circ$. Hence the curve $C$ is defined by $\langle b_{10}, b_{11}, b_{8}b_{9}\rangle$ over the open chart $V_{10}$. Note that $L$ and $Q$ properly intersects at the point $p_{10}$.

(2) Let $L\cap Q=\{p:=p_{10}\}$ be the intersection point of $L$ and $Q$. By tensoring the normal bundle $N_{C/X}$ into the structure sequence $\ses{\cO_{Q}(-p)}{\cO_{C}}{\cO_L}$, we have an exact sequence
\[
\ses{N_{C/X}|_Q(-p)}{N_{C/X}}{N_{C/X}|_L}. 
\]
Since $\rH^1(N_{C/X})=0$, we have $\rH^1(N_{C/X}|_L)=0$. On the other hand, combining with an exact sequence (\cite[Corollary 3.2]{HH83})
\[\ses{N_{L/X}}{N_{C/X}|_L}{\CC_p}\] and $N_{L/X}=\cO_L(1)\oplus \cO_L(-2)$ ((2) of Proposition \ref{linesinx}), we obtain the result.

(3) By doing the same calculation as did the proof of Proposition \ref{deftrip}, the deformation space of the curve $C$ around $V_{10}$ is generated by the columns of the following matrix:
\[
\begin{pmatrix}
b_9&1&0\\
0&0&0\\
0&0&b_8
\end{pmatrix}
\]
with respect to the ordered generators $\{b_{10}, b_{11}, b_{8}b_{9}\}$, which can be extended to the global deformations. By counting the weights of these generators, we obtain the result.
\end{proof}

\begin{remark}
By slightly modifying the proof of the previous two propositions, one can prove that the Hilbert scheme $\bH_3(X)$ is smooth without using the exceptional collections of $X$ (Proposition \ref{h3sm}).
\end{remark}
\subsection{ $\CC^*$-fixed degree $4$ irreducible rational curves in $X$}\label{42}
From now on, we will study the Hilbert scheme $\bH_4(X)$. Since degenerated quartic curves are clearly supported on lower degree curves, we firstly seek $\CC^*$-fixed irreducible rational quartic curves in $X$ through the flop diagram in Proposition \ref{flopd}.

Let $\cU$ be a vector bundle on a smooth projective variety $Z$. Let us call by $\widetilde{C}$ a \emph{lifting} of a curve $C\subset Z$ of the projection map $\PP(\cU)\lr Z$ whenever $\widetilde{C}$ is a subscheme of $\widetilde{C}\subset\PP(\cU)$ given a section of the restricted projection map $\PP(\cU|_C)\lr C$.
\begin{proposition}\label{irrcomli}
Let $C$ be an irreducible quartic rational curve in $X$. Then there exists a lifting $\widetilde{C}$ of $C$ in $\PP(\cE)$ such that its image on $\PP^3$ along the flop map in Proposition \ref{flopd} is a line.
\end{proposition}
\begin{proof}
Let $C$ be an irreducible quartic rational curve in $X$. First of all, we note that $C$ is not contained in $S^{+}$ since the normalization $\PP^1\times \PP^1\lr S^{+}$ is given by a sub-linear system $|\cO(1,11)|$ (\cite[Lemma 6.1]{MU83}). On the other hand, by \cite[Corollary 3.9]{CFK23}, $C\cong \PP^1$ and $\cE|_{C}$ is a vector bundle of rank $2$ and degree $-4$ on the curve $C$. We may let $\cE|_{C}\cong \Oo_{C}(-2-m)\oplus\Oo_{C}(-2+m)$ for some $m\geq 0$. Let $\widetilde{C}$ be the section of the canonical map $\PP(\cE|_C)\to C$ corresponding to the surjection $\cE|_{C}\dual\twoheadrightarrow\Oo_{C}(2-m)$. Let us denote by $e$ the restriction of the hyperplane class of the projective bundle $\PP(\cE)$. Then $e.[\widetilde{C}]=2-m$ by its construction.
Let $\overline{C}$ be the strict transform of $\widetilde{C}$ along the projection map $\pi_{-}$ followed by the flop map $q$. Then the ampleness of $h$ on $\PP^3$ implies that
\[
[\overline{C}].h
=[\widetilde{C}].(3e-H-d)
=3\times (2-m)-4-d.[\widetilde{C}]=2-3m-d.[\widetilde{C}]\geq 0.
\]
But since $C\not\subset S^{+}$, we have $d.[\widetilde{C}]\geq 1$. Therefore $m=0$ and $d.[\widetilde{C}]=1$ or $2$. In the latter case, $[\overline{C}].h=0$ and thus the image of $\widetilde{C}$ along the flop map $q$ must be contained in a fiber of the projection map $\pi_{-}$, which contradicts to Lemma \ref{setmapprpe}. After all, $[\overline{C}].h=1$ and thus we proved the claim.
\end{proof}
From the result of Proposition \ref{irrcomli}, to find a $\CC^*$-fixed irreducible quartic rational curve in $X$, it is enough to check whether various transformed curve of six coordinate lines in $\PP^3$ becomes an irreducible quartic curve $X$. By a numerical class calculation similar to Proposition \ref{irrcomli}, it is enough to consider four fixed lines: $l_{1,-1}$, $l_{3,-1}$, $l_{-1,-3}$, and $l_{3,-3}$ (cf. Example \ref{kjump}). Note that the four lines are $0$, $1$, $2$, $0$-jumping lines respectively.
\begin{proposition}\label{fam1-1}
Let $l:=l_{1,-1}$ be the line supported on the $1$ (and $-1$) coordinates of $\PP^3$. Then the canonical lifting does not provide any irreducible rational quartic curve in $X$.
\end{proposition}
\begin{proof}
From the construction of the flop diagram in Proposition \ref{flopd}, the image $\PP(\cR|_l)\cong \PP^1\times \PP^1$ by the composition map $\pi_{-}\circ q$ in $X$ is exactly the image of $\cC|_l$ except the fibers of two points $p_1$ and $p_{-1}\in \PP^3$. So, among various lifting of the line $l$, the lifting passing through the base point may provide an irreducible quartic curve in $X$. Hence we explicitly construct a family of rational cubic curves in $X$ parameterized by $l$ whose image is exactly $\cC|_{l}$. From the local equation $\cC$ and $X$, we can construct a family of twisted cubics in $X$ such that the central fiber is a degree two map with one base point as follows.
\begin{equation}\label{eq-l1-1}
  P_3(a,t)=\begin{pmatrix}a & 0 & 0 & -\frac{3}{5} t^2& \frac{9}{5}t& -\frac{18}{5}& 0\\
  -\frac{2}{3} at & a & 0& 0& 0 & -\frac{3}{5} t & \frac{2}{5} \\
  0& -2at^2 & at & -\frac{1}{3}a & 0 & 0& \frac{1}{5} t^2 
  \end{pmatrix},
\end{equation}
where $[0:1:t:0]\in l$ and $[1:a]\in \PP^1$. Not that for each $t=t_0\neq 0$, the curve $P_3(a, t_0)$ is a twisted cubic parameterized by $a$. For $t=0$, the curve $P_3(a, 0)$ parameterized by $a$ is a conic with a base point at $a=0$. Hence the possible quartic curve is given by the canonically lifted section $a=0$. But by plugging $a=0$ into \eqref{eq-l1-1}, one can easily see that the image of the curve $P_3(0,t)$ in $X$ is a line and thus it does not give rise to an irreducible quartic curve.

On the other hand, let $[0:s:1:0]\in l$ and $[b:1]\in \PP^1$ be the other chart of the line $l$. By plugging $t=\frac{1}{s}$ and $a=\frac{1}{b}$ in \eqref{eq-l1-1} and doing the elementary operation, we have a family of curves as follows.
\[\begin{pmatrix}
1 & -\frac{3}{2} s & 0 & 0 & 0 & \frac{9}{10} b & -\frac{3}{5} b s\\
0 & 1 & -\frac{1}{2} s & \frac{1}{6} s^2 & 0 & 0 & -\frac{1}{10} b\\
\frac{4}{3} s^2 & 0 & 0 & -\frac{4}{5} b & \frac{12}{5} b s & -\frac{24}{5} b s^2 & 0
\end{pmatrix}.\]
By the same argument as did in the first paragraph, we finish the proof of claim.
\end{proof}
\begin{proposition}
Let $l:=l_{3,-1}$ be the fixed line supported on the $3$ (and $-1$) coordinates of $\PP^3$. Then the canonical lifting provides a fixed irreducible rational normal sextic curve in $X$.
\end{proposition}
\begin{proof}
$\PP(\cR|_l)\cong \PP(\cO_l(-1)\oplus \cO_l(1))$ in $X$ is exactly the image of $\cC|_l$ except the fibers of two points $p_3$, $p_{-1}$. As did in proposition \ref{fam1-1}, our chart around the point $[0:0:1:0]$ is given by
\begin{equation}\label{eq-l3-1}
\begin{pmatrix}
a& 0& 0& 0& -\frac{1}{3}t^2a& \frac{10}{27}& \frac{10}{27}t^3a\\
0& a& 0& \frac{1}{9}ta& 0& \frac{5}{9}t^2a& -\frac{10}{243}\\
0& 0& -\frac{27}{5}t^2a& 1& \frac{18}{5}t^3a& 0& t^4a
\end{pmatrix}
\end{equation}
where $U_3\times \CC=\{[t:0:1:0]\times [1:a]\mid t, a\in \CC\}\subset \PP(\cR|_l)|_{U_{3}}$. Over $t=0$, one can easily see that the fiber is a conic with a base point at $a=\infty$. Let $U_{-1}\times \CC=\{[1:0:s:0]\times [1:b]\mid s, b\in \CC\}\subset \PP(\cR|_l)|_{U_{-1}}$ be other chart of the projective bundle $\PP(R|_l)$. By plugging $t=\frac{1}{s}$ and $a=s^{-2}b$ in \eqref{eq-l3-1}, we obtain
\[\begin{pmatrix}s^3b& 0& 0& 0& -\frac{1}{3}sb& \frac{10}{27}s& \frac{10}{27}b\\
0& s^2b& 0& \frac{1}{9}sb& 0&  \frac{5}{9}b& -\frac{10}{243}\\
0& 0& -\frac{27}{5}s^2b& s^2& \frac{18}{5}sb& 0& b
\end{pmatrix}.
\]
%

For $s=0$, the family parameterizes a fixed point for any $b\neq 0$. For the section $b=\infty$, the parameterized curve is given by
\[
\begin{pmatrix}
1 & 0 & 2s & 0 & -\frac{5}{3}s^2 & 0 & 0\\
0 & 1 & 0 & \frac{1}{9}s & 0 & \frac{5}{9}s^2 & 0\\
0 & 0 & -\frac{27}{5} & 0 & \frac{18}{5}s & 0 & s^2
\end{pmatrix}
\]
which is a degree $6$ rational curve fixed by the $\CC^*$-action.
\end{proof}


\begin{remark}
In fact, for the case of the line $l=l_{3,-1}$, during the strict transform of $\PP(\cR|_{l})$ along the blow-up map $\cM\lr \PP(\cR)$ in Proposition \ref{flopd}, the configuration of fibers and lifting curve is provided by the picture
(FIGURE \ref{fig:Kim3-1}.).
\begin{figure}
\begin{tikzpicture}

\draw[thick,color=blue] (0.5,3.15)--(4.5,3.15);
\draw[thick,color=green] (0.5,0.7)--(4.5,0.7);

\draw[color=purple] (1,4) to[out=-90,in=60] (0.5,1.9);
\draw[color=brown] (0.5,2.1) to[out=-60,in=90] (1,0);

\draw[color=purple] (4,4) to[out=-90,in=120] (4.5,1.9);
\draw[color=brown] (4.5,2.1) to[out=240,in=90] (4,0);

\draw[color=red] (0.7,1.6) to[out=-45,in=-135] (1.6,1.6);
\draw[color=orange] (1.4,1.6) to[out=-45,in=-135] (2.3,1.6);

\node at (0.5,2.6) {\footnotesize $2$};
\node at (0.5,1.2) {\footnotesize $0$};

\node at (4.5,2.6) {\footnotesize $2$};
\node at (4.5,1.2) {\footnotesize $1$};

\node at (1.15,1.65) {\footnotesize $0$};
\node at (1.9,1.65) {\footnotesize $1$};

\node at (2.5,3.5) {\footnotesize $2$};
\node at (2.5,0.5) {\footnotesize $6$};

\node[circle,fill=black,inner sep=0pt,minimum size=3pt] at (0.95,3.15) {};
\node[circle,fill=black,inner sep=0pt,minimum size=3pt] at (0.95,0.7) {};
\node[circle,fill=black,inner sep=0pt,minimum size=3pt] at (4.05,3.15) {};
\node[circle,fill=black,inner sep=0pt,minimum size=3pt] at (4.05,0.7) {};
\node[circle,fill=black,inner sep=0pt,minimum size=3pt] at (4.45,2) {};
\node[circle,fill=black,inner sep=0pt,minimum size=3pt] at (0.55,2) {};

\draw[very thick] (0,0) rectangle (5,4);
\node at (2.5,-0.5) {$\widetilde{\mathcal{P}}$};

\end{tikzpicture}
\vspace{-1em}
\caption{A family of twisted cubics over $l_{3,-1}$} \label{fig:Kim3-1}
\end{figure}
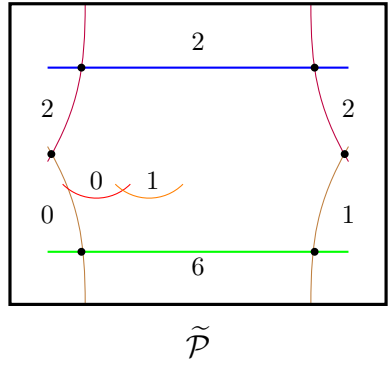
\end{remark}

\begin{proposition}
Let $l:=l_{-1,-3}$ be the fixed line supported on the $3$ (and $1$) coordinates of $\PP^3$.
Then there is no lifting which provides an irreducible rational quartic curve in $X$.
\end{proposition}
\begin{proof}
Since $l=l_{-3,-1}$ is a tangent line to $C_{\text{red}}$, we have $l\subseteq\Sigma^-$ and $\Rr|_{l}\cong\Oo_{\PP^1}(-2)\oplus\Oo_{\PP^1}(2)$.
In this case, $\mathcal{P}\cong\FF_4$ and we denote by $\ell_0$ the canonical section and $f$ the fiber class of $\mathcal{P}$.
Note that $\mathcal{P}\cap\PP(\cS^-)=\ell_0$.

Let $\ell$ be an irreducible section of $\mathcal{P}$ of the linear class $\ell_0+bf$ and $\overline{\ell}$ be its image on $X$ along the flop map in Proposition \ref{flopd}.
Since ${\hat\pi_{-}}^{-1}(\ell_0)$ is contracted via the flop diagram, we may assume that $\ell$ is distinct from $\ell_0$.
Then the degree of $\overline{\ell}$ is given by
\[
[\overline{\ell}].H
=[\ell].(3r+5h-d)
=3\times (b-2)+5-[\ell].\ell_0
=2b+3,
\]
but this number is never equal to $4$.
\end{proof}

\begin{proposition}\label{fixdirr4}
Let $l_{3,-3}$ be the fixed line supported on the $3$ (and $1$) coordinates of $\PP^3$. There uniquely exists a $\CC^*$-fixed quartic rational curve.
\end{proposition}
\begin{proof}
Since the method is similar to previous cases, we only address the conclusion. That is, the family of twisted cubic curves in $X$ over $[t:0:0:1]\in l_{3,-3}$ is
\[\begin{pmatrix}1&0& 0& a_1& -\frac{9}{5}ta_9& 9ta_5& -t^2-ta_1\\0& 1& 0& a_5& -2a_1& 3ta_9& -ta_5\\0& 0& 1& a_9& -6a_5& 10a_1& -ta_9
\end{pmatrix}
\]
where 
\[\det\begin{pmatrix}a_5&5a_1+t& 3ta_9\\-2a_9&-30a_5&-20a_1+5t
\end{pmatrix}=0.\]
The fixed irreducible quartic curve in $X$ is 
\begin{equation}\label{fixedquart}
\begin{pmatrix}-4& 0& 0& -t& 0& 0& 5t^2\\0& -4& 0& 0& 2t& 0& 0\\0& 0& -4& 0& 0& -10t& 0
\end{pmatrix}.\qedhere
\end{equation}

\end{proof}

\begin{remark}
From the calculation in this subsection, one can read the parametric presentation of the fixed lines and conic.
In special, for $[s:t]\in \PP^1$, the fixed conic is 
\begin{equation}\label{fixedcon}
\begin{pmatrix}s&0&0&0&0&-9t&0\\0&s&0&0&0&0&t\\0&0&0&1&0&0&0\end{pmatrix}.
\end{equation}
 \end{remark}
After all, we have a configuration of $\CC^*$-fixed reduced rational curve of degree $\leq 4$ as in the FIGURE \ref{fig:M1}.
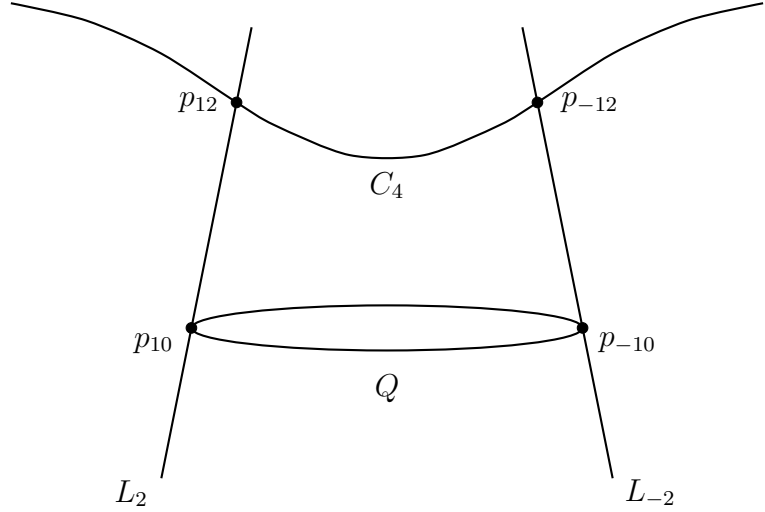
\begin{figure}
  \centering
    \begin{tikzpicture}
      \coordinate (P12) at (-2,2);   
      \coordinate (P-12) at (2,2);   
      \coordinate (P10) at (-2.6,-1);  
      \coordinate (P-10) at (2.6,-1);  
  
      \draw[thick] (-3,-3) -- (P12) -- (-1.8,3); 
      \draw[thick] (3,-3) -- (P-12) -- (1.8,3) ; 
      \node at (-3.4,-3.2) {$L_2$};
      \node at (3.5,-3.2) {$L_{-2}$};
   
      \draw[thick, smooth] plot coordinates {(-5,3.32) (-4,3.09) (-3,2.65) (-2, 2) (-1.5, 1.7) (-0.5, 1.3) (0.5, 1.3) (1.5, 1.7) (2, 2) (3,2.65) (4,3.09) (5,3.32)};
      \node at (0, 0.9) {$C_4$}; 
  
      \draw[thick] (-2.6,-1) arc[start angle=180, end angle=360, x radius=2.6cm, y radius=0.3cm];  
      \draw[thick] (2.6,-1) arc[start angle=0, end angle=180, x radius=2.6cm, y radius=0.3cm];    
      \node at (0, -1.8) {$Q$};  
  
      \filldraw (P12) circle (2pt); 
      \filldraw (P-12) circle (2pt);
      \filldraw (P10) circle (2pt);
      \filldraw (P-10) circle (2pt);
      \node at (-2.5,2) {$p_{12}$};
      \node at (2.7,2) {$p_{-12}$};
      \node at (-3.1,-1.2) {$p_{10}$};
      \node at (3.2,-1.2) {$p_{-10}$};
  \end{tikzpicture}
  \caption{$\CC^*$-fixed reducible rational curves} \label{fig:M1}
  \end{figure}

\subsection{Reducible $\CC^*$-fixed curves and Poincar\'e polynomial of $\bH_4(X)$}\label{trofixpoin}
In this subsection, we compute the weights of the reduced $\CC^*$-fixed rational quartic curve in $X$.
\begin{proposition}\label{irrq}
Let $C_4$ be the $\CC^*$-fixed, irreducible quartic curve in Proposition \ref{fixdirr4}. Then the weights of the tangent space of $\bH_3(X)$ at $[C_4]$ are $4$, $2$, $-2$, and $-4$.
\end{proposition}
\begin{proof}
It is well-known that  $N_{C_4/X}=\cO_{C_4}(1)\oplus \cO_{C_4}(1)\cong \cO_{C_4}(1)\otimes V_2$ for some vector space $V_2$ of dimension $2$, it is enough to find the weights of $\CC^*$-action on each space of the normal bundle. By using the defining equation of $C$ (cf. \eqref{fixedquart}), one can easily see that the weights of $V_2$ are $\pm1$ and the  $\CC^*$-action on $\rH^0(\cO_{C_4}(1))\cong \CC^2$ is given by $\pm 3$. Thus one can finish the proof.
\end{proof}

\begin{proposition}\label{req}
Let $C=L_{2}\cup L_{-2} \cup Q$ be the $\CC^*$-fixed, reduced quartic curve in $X$ such that $L_{\pm2}$ are fixed lines and $Q$ is the fixed smooth conic (cf. FIGURE \ref{fig:M1}). Then the weights of the tangent space of $\bH_4(X)$ at the point $[C]$ are $4$, $2$, $-2$, and $-4$.
\end{proposition}
\begin{proof}
Since the curve $C$ has the $A_1$-singularity at two singular points $p_{10}$ and $p_{-10}$, the normal sheaf $N_{C/X}$ is a locally free sheaf of rank two. Hence from the structure sequence $\ses{\cO_{Q}(-p_{10}-q_{-10})}{\cO_{C}}{\cO_{L_2}\oplus\cO_{L_{-2}}}$, we have
\begin{equation}\label{llcex}
\ses{N_{C/X}|_{Q}(-p_{10}-q_{-10})}{N_{C/X}}{N_{C/X}|_{L_2\sqcup L_{-2}}\cong N_{C/X}|_{L_2}\oplus N_{C/X}|_{L_{-2}}}.
\end{equation}
Note that the $\CC^*$-action on the direct sum part of the last term in \eqref{llcex} trivially acts since it comes from the restriction map $\cO_{C}\twoheadrightarrow \cO_{L_{i}}, i=\pm2$.
But from the short exact sequence $\ses{N_{Q/X}}{N_{C/X}|_{Q}}{\CC_{p_{10}}\oplus\CC_{p_{-10}}}$, $N_{Q/X}\cong \cO_{Q}\oplus\cO_{Q}$ (Proposition \ref{conicsinx}) and the isomorphism in \eqref{lqn}, we see that $N_{C/X}|_{Q}\cong \cO_Q(p_{10})\oplus \cO_Q(p_{-10})$ (cf. \cite[Corollary 3.2]{HH83}) and $N_{C/X}|_{L_{i}}=N_{L_i\cup Q/X}|_{L_{i}}\cong \cO_{L_i}(1)\oplus \cO_{L_i}(-1)$ for $i=\pm2$. Plugging these data in \eqref{llcex}, we conclude that $\rH^1(N_{C/X})=0$ and
\begin{equation}\label{llqnor}
\rH^0(N_{C/X})=\rH^0(N_{C/X}|_{L_2})\oplus\rH^0(N_{C/X}|_{L_{-2}})=\rH^0(\cO_{L_2}(1))\oplus\rH^0(\cO_{L_{-2}}(1)).
\end{equation}
From the proof of (2) and (3) of Proposition \ref{smlq}, the weights of the first (resp. second) term of the rightmost side in \eqref{llqnor} are $4$, $2$ (resp. $-4$, $-2$) (cf. Proposition \ref{lacx}).
\end{proof}
In general, the Hilbert scheme parameterizes non Cohen-Macaulay connected (briefly CM) curves. But in our case, since $X\subset \PP^{34}$ are generated by quadratic and linear forms and does not contain a plane, this cannot occur as follows.
\begin{proposition}\label{cucm}
Every connected curve $C$ with Hilbert polynomial $4m+1$ in $X$ is a CM-curve.
\end{proposition}
\begin{proof}
Let $\bar{C}$ be the maximal CM-subcurve of the curve $C$. By its definition $\deg(\bar{C})=4$. If the curve $\bar{C}$ is planar, then $X$ contains the plane $\langle \bar{C}\rangle =\PP^2$ which violates the fact that $\dim \bH_1(X)=1$. Hence by the degree-genus formula, the possible arithmetic genus of $\bar{C}$ are $p_a(\bar{C})=0$ or $1$. When $p_a(\bar{C})=1$, then it is well-known that $\langle\bar{C}\rangle=\PP^3$. Also by \cite[Corollary 3.9]{CFK23}, the curve $\bar{C}$ is not a local complete intersection. Then, the possible generators of the defining equation $\bar{C}$ was classified in \cite[Theorem 5.2]{AV92} (cf. \cite[(2) and (3) of Proposition 2.5]{CCM16}). That is, the ideal $\bar{C}$ consists of two quadrics having a common linear factor and a cubic form, then it violates that $X$ does not contain any plane again.
\end{proof}
Note that the structure sheaf of the CM-curve is stable (Proposition \ref{stabr}). Since the MU-variety does not contain a pair of lines, the reduced support of fixed curves are of types: A line, conic, and a pair of line and conic. For the detailed introduction and properties of the CM-filtration of CM-curve in a smooth projective variety of dimension $3$, see \cite{BF87, NS03}.
\begin{proposition}\label{mult4line}
\begin{enumerate}
\item An CM-curve $C$ in $X$ supported on the $\CC^*$-fixed line $L$ is unique.
\item The Hilbert scheme $\bH_4(X)$ at $[C]$ is smooth with weights $6$, $4$, $4$, and $2$.
\end{enumerate}
\end{proposition}
\begin{proof}
(1) If the curve $C$ has an embedding dimension $3$ at each point,
then $I_{L}^3\subset I_C\subset I_{L}^2$ (\cite[Section 4]{BF87}). Also $I_L^2/I_C\cong \cO_L(-4)$ because $I_L/I_L^2\cong \cO_L(-1)\oplus \cO_L(2)$. Combining these ones, we have the structure sequence
\begin{equation}\label{thick4}
\ses{I_C/I_L^3}{I_L^2/I_L^3\cong\Sym^2(\cO_L(-1)\oplus \cO_L(2))}{I_L^2/I_C=\cO_L(-4)}
\end{equation}
which is impossible because of the degree reasoning of the third map in \eqref{thick4}. If the curve $C$ has generic embedding dimension $2$, there exists a CM-filtration of $C$:
\[
L\subset D\subset W\subset C
\]
such that $I_{L|D}\cong \cO_L(a)$, $I_{D|W}\cong \cO_L(2a+b)$ and $I_{W|C}\cong \cO_L(3a+c)$. By the construction of CM-filteration of $C$ and Proposition \ref{linesinx}, we see that $a\geq -1$, $0\leq b\leq c$. Since $p_a(C)=0$, we have $6a+b+c=-3$. Hence the possible types are given by $(a, b, c)=(-1,0,3)$ or $(-1, 1, 2)$  (\cite[Section 2]{NS03}). For the first case, $\chi(\cO_W(m))=3m$, which implies that the subcurve $W$ is planar by the degree-genus formula. Since the defining equations of $X$ are generated by quadrics and linear forms, $\langle W\rangle =\PP^2\subset X$, which is a contradiction. Hence the CM-curve $C$ has a CM-subcurve $W$ (resp. $D$) with the Hilbert polynomial $3m+1$ (resp. $2m+1$).
We claim that the CM-curve $C$ with this filtration is unique if it exists. By a diagram chase and Proposition \ref{stabr}, one can see that $\cO_C$ lies in the exact sequence
\[
\ses{G}{\cO_C}{\cO_{D}}
\]
such that $G$ fits into the non-split extension $\ses{\cO_L(-1)}{G}{\cO_L(-1)}$. But such non-extensions are parameterized by a $\CC^{a-1}$-bundle over $\PP^{a-1}$ for $a=\dim\Ext^1(\cO_{D},\cO_L(-1))$ (\cite[Proposition 4.4]{Mun08}). Since $a=1$, the possible CM-curve $C$ is unique. Now we find the defining ideal of the unique curve $C$. From the structure sequence: $\ses{\cO_L(-1)}{\cO_{W}}{\cO_{D}}$, we have a surjective map: $I_{W}\twoheadrightarrow \cO_L(-1)$. By taking the derived functor $\Hom(I_{W},-)$, we obtain an exact sequence:
\[
0\lr \Hom(I_{W},\cO_L(-1))\lr \Hom(I_{W},\cO_{W})\stackrel{\Psi}{\lr} \Hom(I_{W},\cO_{D})\lr \cdots
\]
where the middle term is the deformation space of $\bH_3(X)$ at $[W]$.
Among the generators in \eqref{genssec3}, the last column only generates the kernel of $\Psi$. So we have a unique surjective map $\phi_3: I_{W}\twoheadrightarrow \cO_L(-1)$, $\phi_3\begin{pmatrix} a_4\\a_8\\ a_{12}\end{pmatrix}=\begin{pmatrix} 0\\0\\a_{11}\end{pmatrix}$ and thus the kernel consists of $\langle a_4, a_8, a_{10}a_{12}\rangle$. Therefore the ideal $I_{C}$ is given by
\[\begin{split}
\langle a_8, a_7+3a_{12}, 5a_6+a_{11}, 6a_5+a_{10}, a_4, a_3, &5a_2-9a_{12}, 10a_1-a_{11}, a_{10}a_{12}, a_{11}^2,\\ a_{10}a_{11}-18a_9a_{12}, &5a_{10}^2-6a_9a_{11}+12a_{12} \rangle.
\end{split}\]
Using the Macaulay2 program (\cite{M2}), one can check that the projectivization of this curve is exactly the multiplicity $4$-line in $X$.

(2) On the open chart $V_{10}$, the ideal of the curve $C$ is generated by $\langle b_4, b_8, b_{10}b_{12}\rangle$. As did in the degree three case, by considering the weight decomposition of the deformation space, one can see that the column of the matrix:
\[
\begin{pmatrix}
0&0&0&\frac{3}{2}b_{12}\\
0&0&b_{12}&-\frac{1}{6}b_{11}\\
b_{11}&b_{12}&-6b_9 b_{12}&b_9b_{11}+6b_{10}
\end{pmatrix}
\]
generates the sheaf $\cH om (I_C, \cO_{C})$ over $V_{10}$ where the weights are given by $4$, $2$, $4$, and $6$.
\end{proof}

\begin{remark}
As did in \cite[Section 3.3]{CK11}, one can construct a flat family of sheaves over the stable maps space. By performing the elementary modification of sheaves, one can obtain a stable limit. In our setting, for the general $4:1$-covering $f:\PP^1\lr L\subset X$, the direct image sheaf of $\cO_{\PP^1}$ by the map $f$ is of the form $f_*{\cO_{\PP^1}}=\cO_L\oplus \cO_L(-1)^{\oplus3}$. Hence the possible type of stable sheaf (CM-curve) is of the form in Proposition \ref{mult4line}.
\end{remark}

\begin{lemma}\label{nestedseq}
Let $Y \stackrel{i}{\hookrightarrow} X$ be a smooth, closed subvariety of the smooth variety $X$.
If $F$ and $G \in \mathrm{Coh}(Y)$, then there is an exact sequence
\begin{equation}\label{thomas2}
\begin{split}
0 \to \Ext^1_{Y}(F, G) \to \Ext^1_{X}(i_*F, i_*G) &\to \Hom_{Y}(F,G\otimes N_{Y/X}) \\
&\to \Ext^2_{Y}(F, G) \to \Ext^2_{X}(i_*F, i_*G).
\end{split}
\end{equation}
\end{lemma}
\begin{proof}
This is the base change spectral sequence in \cite[Theorem 12.1]{MCC00}.
\end{proof}

\begin{proposition}
Let $C$ be a smooth conic in $X$. Then the extension class is
\[
\dim \Ext_X^1(\cO_C,\cO_C(-p))=0
\]
for some $p\in C$. Hence there does not exist a double structure on a smooth conic in $X$.
\end{proposition}
\begin{proof}
By Proposition \ref{conicsinx} and Lemma \ref{nestedseq},
\[
0=\rH^1(O_C(-p))=\Ext_C^1(\cO_C,\cO_C(-p))\lr \Ext_X^1(\cO_C,\cO_C(-p)) \lr \rH^0(N_{C/V}\otimes O_C(-p))=0
\]
we proved the claim. Let $C^2$ be a double structure on $C$ with Hilbert polynomial $4m+1$. Let $\cO_{C^2}\twoheadrightarrow \cO_C$ be the restriction map. The kernel $\cK$ of the restriction map is torsion-free and thus it is locally free of rank one on $C$. Thus it is of the form $\cK\cong \cO_C(-p)$. Clearly, it is not a split type because of the stability of the structure sheaf $\cO_{C^2}$ (Proposition \ref{stabr}).
\end{proof}

\begin{proposition}\label{l2q}
Let $C=L^2\cup Q$ be a quartic curve with a double structure on the line $L$ and a smooth conic $Q$. 
\begin{enumerate}
\item Then $C$ is unique. 
\item Furthermore, the Hilbert scheme $\bH_4(X)$ is smooth at $[C]$ with weights $6$, $2$, $2$, $-2$.
\end{enumerate}
\end{proposition}
\begin{proof}
(1): The curve $C$ is determined by the extension class $F$ which fits into $\ses{\cO_L(-1)}{F}{\cO_{L\cup Q}}$. But since $\rH^0(N_{L\cup Q/X}|_{L}(-1))=\CC$, we have a unique such a curve $C$.

(2): Let us define the ideal of $C$ by the union of the double line and conic. Since the planes containing the double line (defined by the ideal $\langle b_1, b_2, b_3, b_4, b_6, b_7, b_8, b_{11}, b_{12}, 6b_5-b_{10}, b_{10}^2\rangle$ in $V_{10}^\circ$)  and the conic (defined in equation \eqref{fixedcon}) meets at point, it is enough to intersect two ideals. After all, over the open chart $V_{10}$ it is generated by $\langle b_4, b_{11}, b_8b_9\rangle$. Then the deformation space is generated the three columns with respect to this basis:
\[
\begin{pmatrix}
0&0&0&0\\
0&0&b_{10}&b_9\\
b_{10}&b_8&0&0
\end{pmatrix}.
\]
One can easily see that this space can be extended into the open chart $V_{12}$ and $V_{-10}$. By weight computation, we proved the claim.
\end{proof}
One can compute the Poincar\'e polynomial of the Hilbert scheme $\bH_4(X)$ by using Theorem \ref{prop:onetoone}. In general, let $\bM$ be a $\CC^*$-invariant closed subscheme in a projective space $\PP^r$. If the space $\bM$ is smooth for each point in the $\CC^*$-fixed locus $\bM^{\CC^*}$, then $\bM$ is a smooth variety.
\begin{proposition}\label{mainpr}
Let $X$ be the Mukai-Umemura variety. Then, $\bH_4(X)$ is smooth and
$P(\bH_4(X))=1+p+2p^2+p^3+p^4$.
\end{proposition}
\begin{proof}
From Proposition \ref{irrq}, \ref{req}, \ref{mult4line}, \ref{l2q}, we know that $\bH_4(X)$ is smooth. Hence we obtain the Poincar\'e polynomial of $\bH_4(X)$ by Theorem \ref{prop:onetoone} and symmetry.
\end{proof}

\bibliographystyle{alpha}
\bibliography{Library.bib}

\end{document}